\documentclass{amsart}
\usepackage{graphics,psfrag}
\usepackage{pgf,tikz}
\usetikzlibrary{arrows}
\usepackage{hyperref}

\DeclareMathOperator{\val}{val}
\DeclareMathOperator{\rep}{rep}

\theoremstyle{theorem}
\newtheorem{theorem}{Theorem}
\newtheorem{lemma}[theorem]{Lemma}
\newtheorem{corollary}[theorem]{Corollary}

\theoremstyle{definition}
\newtheorem{definition}[theorem]{Definition}
\newtheorem{example}[theorem]{Example}
\theoremstyle{remark}
\newtheorem{remark}[theorem]{Remark}

\begin{document}
\title{Generalized Pascal triangle for binomial coefficients of words}
\author{Julien LEROY}
\author{Michel RIGO}
\author{Manon STIPULANTI}
\thanks{M. Stipulanti is supported by a FRIA grant.}
\address{Universit\'e de Li\`ege, Institut de math\'ematique, All\'ee de la d\'ecouverte 12 (B37),
4000 Li\`ege, Belgium\newline
J.Leroy@ulg.ac.be, M.Rigo@ulg.ac.be, M.Stipulanti@ulg.ac.be}

\begin{abstract}
    We introduce a generalization of Pascal triangle based on binomial coefficients of finite words. These coefficients count the number of times a word appears as a subsequence of another finite word. Similarly to the Sierpi\'nski gasket that can be built as the limit set, for the Hausdorff distance, of a convergent sequence of normalized compact blocks extracted from Pascal triangle modulo~$2$, we describe and study the first properties of the subset of $[0,1]\times [0,1]$ associated with this extended Pascal triangle modulo a prime $p$.
\end{abstract}

\maketitle
\section{Introduction}

Pascal triangle and the corresponding Sierpi\'nski gasket are well studied objects (see, for instance, \cite{stewart} for a survey). They exhibit self-similarity features and have connections with dynamical systems, cellular automata, number theoretic questions and automatic sequences \cite{AB,Adam,CLR,vonH,Mauldin}.
In this paper, we will consider a variation of these two objects by extending binomial coefficients to the free monoid $A^*$ where $A$ is a finite alphabet. 

Let us start with basic combinatorial definitions. A {\em finite word} is simply a finite sequence of letters belonging to a finite set called {\em alphabet}. In combinatorics on words, one can introduce 
the binomial coefficient $\binom{u}{v}$ of two finite words $u$ and $v$ which is
the number of times $v$ occurs as a subsequence of $u$ (meaning as a
``scattered'' subword). As an example, we consider two particular words over $\{0,1\}$ and 
$$\binom{101001}{101}=6.$$
Indeed, if we index the letters of the first word $u_1u_2\cdots u_6=101001$, we have 
$$u_1u_2u_3=u_1u_2u_6=u_1u_4u_6=u_1u_5u_6=u_3u_4u_6=u_3u_5u_6=101.$$
Observe that this concept is a natural generalization of the binomial coefficients of integers. For a single letter alphabet $\{a\}$, we have
\begin{equation}
    \label{eq:1lettre}
\binom{a^m}{a^n}=\binom{m}{n},\quad \forall\, m,n\in\mathbb{N}
\end{equation}
where $a^m$ denotes the concatenation of $m$ letters $a$. For more on these binomial coefficients, see for instance \cite[Chap.~6]{Lot}.

In this paper, we are interested in Pascal triangle obtained when considering binomial coefficients of words. To define such a triangular array, we will consider all the words over a finite alphabet and we order them by genealogical ordering (i.e., first by length, then by the classical lexicographic ordering for words of the same length assuming $0<1$). For the sake of simplicity, we will mostly discuss the case of a $2$-letter alphabet $\{0,1\}$ and to relate these words to base-$2$ expansions of integers, we will assume without loss of generality that the non-empty words start with $1$. (If leading zeroes were allowed, then different words could represent the same integer.)

\begin{definition}
We let $\rep_2(n)$ denote the greedy base-$2$ expansion of $n\in\mathbb{N}_{>0}$ starting with $1$ where the notation $\mathbb{N}_{>0}$ stands for the set of all positive integers. We set $\rep_2(0)$ to be the empty word denoted by $\varepsilon$. Let $L=\{\varepsilon\}\cup 1\{0,1\}^*$ be the set of base-$2$ expansions of the integers. We let $L_n$ denote the set of words of length at most $n$ belonging to $L$, i.e.,
$$L_n=(\{\varepsilon\}\cup 1\{0,1\}^*)\cap \{0,1\}^{\le n}.$$
Note that $\# L_n=2^n$ for all $n\ge 0$. 
\end{definition}
The first few values in the generalized Pascal triangle $T$ that we will deal with are given in Table~\ref{tab:tp}. 
\begin{table}[h!t]
$$\begin{array}{r|cccccccc}
&\varepsilon&1&10&11&100&101&110&111\\
\hline
 \varepsilon&\mathbf{1} & 0 & 0 & 0 & 0 & 0 & 0 & 0 \\
 1&\mathbf{1} & \mathbf{1} & 0 & 0 & 0 & 0 & 0 & 0 \\
 10&1 & 1 & 1 & 0 & 0 & 0 & 0 & 0 \\
 11&\mathbf{1} & \mathbf{2} & 0 & \mathbf{1} & 0 & 0 & 0 & 0 \\
 100&1 & 1 & 2 & 0 & 1 & 0 & 0 & 0 \\
 101&1 & 2 & 1 & 1 & 0 & 1 & 0 & 0 \\
 110&1 & 2 & 2 & 1 & 0 & 0 & 1 & 0 \\
 111&\mathbf{1} & \mathbf{3} & 0 & \mathbf{3} & 0 & 0 & 0 & \mathbf{1} \\
\end{array}$$
\caption{The  first few values in the generalized Pascal triangle.}
    \label{tab:tp}
\end{table}
These values correspond to the words $\varepsilon,1,10,11,100,101,110,111$.

Considering the intersection of the lattice $\mathbb{N}^2$ with the region $[0,2^j]\times [0,2^j]$, the first $2^j$ rows and columns of the usual Pascal triangle $(\binom{m}{n}\bmod{2})_{m,n< 2^j}$ provide a coloring of this lattice. If we normalize this region by a homothety of ratio $1/2^j$, we get a sequence in $[0,1]\times [0,1]$ converging, for the Hausdorff distance, to the Sierpi\'nski gasket when $j$ tends to infinity \cite{vonH}.

It is therefore natural to ask whether a similar phenomenon occurs in our extended framework. The generalized Pascal triangle limited to words in $L_n$ has $2^n$ rows and columns. After coloring and normalization (see Figure~\ref{fig:U3} for a picture of the cases when $n=3,4$ and Figure~\ref{fig:U9} in the appendix for $n=9$), can we expect the convergence to an analogue of the Sierpi\'nski gasket and could we describe the possible limit object? We answer positively and the limit object is described as the topological closure of a union of segments that are described through a simple combinatorial property (see Figure~\ref{fig:approx} in the appendix). For the sake of simplicity, we will mostly describe the coloring modulo $2$. At the end of the paper, we shortly discuss colorings modulo a prime number $p$ (see Figure~\ref{fig:U72} for coefficients congruent to $2$ modulo $3$).

In our construction, we take at each step exactly $2^n$ words and a scaling (or normalization) factor of $1/2^n$. For instance, the authors in \cite{barbe} discussed which sequences can be used as scaling factors for objects related to automatic sequences. In particular, the classical Pascal triangle modulo $p^d$ is shown to be $p$-automatic in \cite{AB}, where $p$ is prime and the scaling sequence has to be of the form $(p^{kn+j})_{n\ge 0}$ with $j=0,\ldots,p-1$.

\begin{remark}
We can make some extra comments about our study. 
Note that, from \eqref{eq:1lettre} the usual Pascal triangle is a ``sub-array'' of our extended triangle $T$ by considering only words of the form $1^m$ and $1^n$. In Table~\ref{tab:tp}, the elements of the usual triangle are written in bold. We can also observe that the second column of this extended triangle $T$ is exactly $(s_2(n))_{n\ge 0}$ where $s_2$ denotes the sum-of-digits function for base-$2$ expansions of integers \cite{delange}. Thus, considering these values modulo $2$, the second column is exactly the well-known Thue--Morse word, see \cite{Pin} for connections between binomial coefficients of words and $p$-adic topology. Moreover, each column of $T$ is related to the language (i.e., the set of finite words) 
$$L(v,r,p):=\left\{u\in A^*\mid \binom{u}{v}\equiv r\bmod{p}\right\},\ v\in A^*, r\in\{0,\ldots,p-1\}.$$
In formal language theory \cite{Eil}, a language is a $p$-group language if and only if it is a Boolean combination of languages $L(v_i,r_j,p)$. For an extension of a theorem of Mahler in $p$-adic analysis to functions defined on the free monoid $A^*$, see \cite{Pin2}.
\end{remark}

This paper is organized as follows. In Section~\ref{sec:2}, we collect some results on binomial coefficients. In Section~\ref{sec:3}, we define a subset $T_n$ of $[0,2^n]\times [0,2^n]$ associated with the parity of the first $2^n$ binomial coefficients. We prove that the subset of non-zero binomial coefficients for words in $L_n$ contains exactly $3^n$ elements. The set $T_n$ is then normalized by a factor $1/2^n$ to give the sequence $(U_n)_{n\ge 0}$ of subsets in $[0,1]\times[0,1]$. In the second part of Section~\ref{sec:3}, we independently build a compact set $\mathcal{A}_0$ from a countable union of segments and an important combinatorial condition $(\star)$. This condition is at the heart of our discussions (and will be generalized in Section~\ref{sec:5} to take into account the situation modulo $p$). From $\mathcal{A}_0$ and using two maps $h$ and $c$, the first one being a homothety, we define a sequence of compact sets naturally converging to a limit compact set $\mathcal{L}$. In Section~\ref{sec:4}, similarly to the construction of the Sierpi\'nski gasket, we prove that $(U_n)_{n\ge 0}$ tends to $\mathcal{L}$. For the sake of simplicity, we have limited our presentation to the case of odd binomial coefficients. In Section~\ref{sec:5}, we briefly sketch the main changes when considering congruences modulo a prime $p$. In the last section, we present some pictures of the studied sets. The limit set $\mathcal{L}$ convinces us that these extended Pascal triangles contain many interesting combinatorial and dynamical questions to consider.

\section{Basic results on binomial coefficients}\label{sec:2}
In the first part of this section, we collect well-known facts on binomial coefficients of words or integers. For a proof of the first two lemmas, we refer the reader to \cite[Chap.~6]{Lot}. 

\begin{lemma}\label{lem:binomial}
Let $u,v$ be two words of $L$ and let $a,b$ be two letters in $\{0,1\}$. Then we have
    $$\binom{ua}{vb}=\binom{u}{vb}+\delta_{a,b}\binom{u}{v}$$
where $\delta_{a,b}$ is equal to $1$ if $a=b$, $0$ otherwise.
\end{lemma}

\begin{lemma}\label{lem:binomial2}
Let $s,t,w$ be three words of $L$. Then we have
    $$\binom{sw}{t}=\sum_{\substack{u,v\in L\\ uv=t}} \binom{s}{u}\binom{w}{v}.$$
\end{lemma}

The main observation of the next lemma is that if $\binom{u}{vb}\equiv 1\bmod{2}$ for some $b\in\{0,1\}$, then there exists $a\in\{0,1\}$ such that $\binom{ua}{vb}\equiv 1\bmod{2}$.
\begin{lemma}\label{lem:carres}
Let $u,v$ be words of $L$. The following tables display the values modulo $2$ of the binomial coefficient $\binom{ua}{vb}$ for all $a,b\in\{0,1\}$ when the values modulo $2$ of the binomial coefficients $\binom{u}{v}$ and $\binom{u}{vb}$ with $b\in\{0,1\}$ are known.
\begin{align*}
&
\begin{array}{c|c|cc}
 &v&v0&v1\\
    \hline
u&0&0&0\\
\hline
u0& &0&0\\
u1& &0&0\\
\end{array} \quad
\begin{array}{c|c|cc}
 &v&v0&v1\\
    \hline
u&0&0&1\\
\hline
u0& &0&1\\
u1& &0&1\\
\end{array} \quad 
\begin{array}{c|c|cc}
 &v&v0&v1\\
    \hline
u&0&1&0\\
\hline
u0& &1&0\\
u1& &1&0\\
\end{array} \quad
\begin{array}{c|c|cc}
 &v&v0&v1\\
    \hline
u&0&1&1\\
\hline
u0& &1&1\\
u1& &1&1\\
\end{array} \\
&
\begin{array}{c|c|cc}
 &v&v0&v1\\
    \hline
u&1&0&0\\
\hline
u0& &1&0\\
u1& &0&1\\
\end{array} \quad
\begin{array}{c|c|cc}
 &v&v0&v1\\
    \hline
u&1&0&1\\
\hline
u0& &1&1\\
u1& &0&0\\
\end{array} \quad 
\begin{array}{c|c|cc}
 &v&v0&v1\\
    \hline
u&1&1&0\\
\hline
u0& &0&0\\
u1& &1&1\\
\end{array} \quad
\begin{array}{c|c|cc}
 &v&v0&v1\\
    \hline
u&1&1&1\\
\hline
u0& &0&1\\
u1& &1&0\\
\end{array}
\end{align*}
\end{lemma}
\begin{proof}
It directly follows from Lemma~\ref{lem:binomial}.
\end{proof}

Let us also recall Lucas' theorem relating classical binomial coefficients modulo a prime $p$ with base-$p$ expansions. See \cite[p. 230]{lucas} or \cite{lucas2}. Note that in the following statement, if the base-$p$ expansions of $m$ and $n$ are not of the same length, then we pad the shorter with leading zeroes.
\begin{theorem}\label{thm:lucas}
Let $m$ and $n$ be two non-negative integers and let $p$ be a prime. If 
$$m=m_kp^k + m_{k-1}p^{k-1} + \cdots + m_1 p + m_0$$ and $$n=n_kp^k + n_{k-1}p^{k-1} + \cdots + n_1 p + n_0$$ with $m_i,n_i\in\{0,\ldots,p-1\}$ for all $i$, then the following congruence relation holds $$\binom{m}{n}\equiv \prod _{i=0}^k \binom {m_i}{n_i} \bmod p,$$
using the following convention: $\binom {m}{n}= 0$ if $m < n$.
\end{theorem}

\section{Graphical representation}\label{sec:3}
We let $w_i$ denote the $i$th word of the language $L$ in the genealogical order. Note that $w_i=\rep_2(i)$ for $i\ge 0$. Considering the intersection of the lattice $\mathbb{N}^2$ with the region $[0,2^n]\times [0,2^n]$, the first $2^n$ rows and columns of the generalized Pascal triangle $$\left(\binom{w_i}{w_j}\bmod{2}\right)_{0\le i,j< 2^n}$$ provide a coloring of this lattice leading to the definition of a sequence $(T_n)_{n\ge 0}$ of subsets of $\mathbb{R}^2$. If we normalize these $T_n$ by a homothety of ratio $1/2^n$, we will define a sequence $(U_n)_{n\ge 0}$ in $[0,1]\times [0,1]$.
\subsection{Definition of the sequences $(T_n)_{n\in\mathbb{N}}$ and $(U_n)_{n\in\mathbb{N}}$}
\begin{definition}\label{def:Tn}
Let $Q:=[0,1]\times[0,1]$. Consider the sequence $(T_n)_{n\ge 0}$ of sets in $\mathbb{R}^2$ defined for all $n\ge 0$ by 

$$
T_n:= \bigcup \left\{(\val_2(v),\val_2(u))+Q\mid u,v\in L_n, \binom{u}{v}\equiv 1\bmod{2}\right\}\subset [0,2^n]\times [0,2^n].
$$
Each $T_n$ is a finite union of unit squares and is thus compact.
\end{definition}

\begin{remark}
In Table~\ref{tab:fv}, we count the number of unit squares in $T_n$ for the first few values of $n$ and we compare this quantity to the number of positive binomial coefficients of pairs of words in $L_n$.
    \begin{table}[h!t]
        $$\begin{array}{r|cccccccccccc}
            &1&2&3&4&5&6&7&8&9&10\\
            \hline
         \# \text{ unit squares}& 3& 8& 22& 62& 166& 458& 1258& 3510& 9838& 27598\\
         \# \text{ positive coefficients}& 3& 9& 27& 81& 243& 729& 2187& 6561& 19683& 59049\\
        \end{array}$$
        \caption{Number of unit squares in $T_n$, $n=1,\ldots,10$.}
        \label{tab:fv}
    \end{table}
Observe that, as shown in the next lemma, the number of positive binomial coefficients of pairs of words in $L_n$ is equal to $3^n$.
\end{remark}

\begin{lemma}
For all $n \in \mathbb{N}$, the number of pairs of words in $L_n$ having a positive binomial coefficient is equal to $3^n$.
\end{lemma}
\begin{proof}
For any positive integer $n$, let $V_n$ denote the set of pairs of integers $(x,y)$ such that $2^{n-1} \leq y < 2^n$, $0 \leq x \leq y$ and $\binom{\rep_2(y)}{\rep_2(x)}>0$.
Thus, $\# V_n$ corresponds to the number of pairs of words in $(L_n\setminus  L_{n-1}) \times L_n$ having a positive binomial coefficient.
We prove that $\# V_n = 2\cdot 3^{n-1}$ by induction on $n \geq 1$. This proves the result since the number of positive binomial coefficients of pairs of words in $L_0$ is $1 = 3^0$.

The result is clear for $n = 1$. Let us suppose it is true up to $n$ and let us prove it for $n+1$.
For all integers $n \geq 1$ and $m \geq 0$, consider the set  
\[
X_{m,n} := 
	\begin{cases}
		\emptyset,
			& \text{if } m>n; \\
		V_n \cap (\{0\} \times \mathbb{N}),
		 	& \text{if } m = 0;	\\
		V_n \cap ([2^{m-1},2^m) \times \mathbb{N}),  
			& \text{otherwise}.
	\end{cases}
\]
We thus have the following partition 
\[
	V_n = \bigcup_{m=0}^n X_{m,n},
\]
with, for $m \geq 1$, 
\[
	X_{m,n} = \left\{(\val_2(v),\val_2(u))\mid u\in L_n\setminus L_{n-1}, v\in L_m\setminus L_{m-1}, \binom{u}{v}>0\right\}.
\]		
Let us define the functions $f_1$, $f_2$, $f_3$ and $f_4$ by
\begin{eqnarray*}
	f_1(x,y) &=& (2x,2y), 			\\ 
	f_2(x,y) &=& (2x+1,2y+1), 	\\
	f_3(x,y) &=& (x,2y),			\\
	f_4(x,y) &=& (x,2y+1).
\end{eqnarray*}
It follows from Lemma~\ref{lem:binomial} that for all $n \geq 1$ and $m \geq 0$, 
\[
	X_{m+1,n+1} = f_1(X_{m,n}) \cup f_2(X_{m,n}) \cup f_3(X_{m+1,n}) \cup f_4(X_{m+1,n}). 
\]
We thus get
\begin{eqnarray*}
	V_{n+1} 
	&=& \bigcup_{m=0}^{n+1} X_{m,n+1}	\\
	&=& \left(\{0\} \times ([2^n,2^{n+1}) \, \cap \, \mathbb{N})\right) \\
	& &	\cup \left( \bigcup_{m=0}^{n} f_1(X_{m,n}) \cup f_2(X_{m,n}) \cup f_3(X_{m+1,n}) \cup f_4(X_{m+1,n}) \right) \\
	&=& f_1(V_n) \cup f_2(V_n) \cup f_3(V_n) \cup f_4(V_n).
\end{eqnarray*}
Observe that $f_1 (V_n) \cap f_2(V_n) = \emptyset$ and $f_3 (V_n) \cap f_4(V_n) = \emptyset$.
Furthermore, if $(x,y) \in V_n$, then exactly one of the two elements $f_3(x,y)$ and $f_4(x,y)$ belongs to $f_1 (V_n) \cup f_2(V_n)$.
Indeed, if $(x,y)$ belongs to $V_n$, then so does $(\lfloor x/2\rfloor,y)$ and exactly one of the following two equalities is satisfied (depending on the parity of $x$):
\[
	f_3(x,y) = f_1(\lfloor x/2\rfloor,y)
	\quad \text{or} \quad 
	f_4(x,y) = f_2(\lfloor x/2\rfloor,y).
\]
We thus get $\# V_{n+1} = 3 \# V_n = 2 \cdot 3^n$, which concludes the proof.
\end{proof}

\begin{definition}\label{def:Un}
We will be interested in the sequence $(U_n)_{n\ge 0}$ of compact sets defined for all $n\ge 0$ by 
$$U_n:=\frac{T_n}{2^n}\subset [0,1]\times[0,1].$$
\end{definition}
In Figure~\ref{fig:U3}, we have depicted the sets $U_3$ and $U_4$. The set $U_9$ is depicted in Figure~\ref{fig:U9} given in the appendix.
\begin{figure}[h!tbp]
    \centering
    \includegraphics{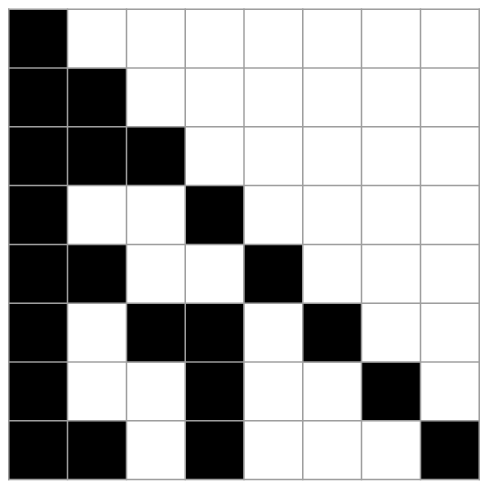}\quad \includegraphics{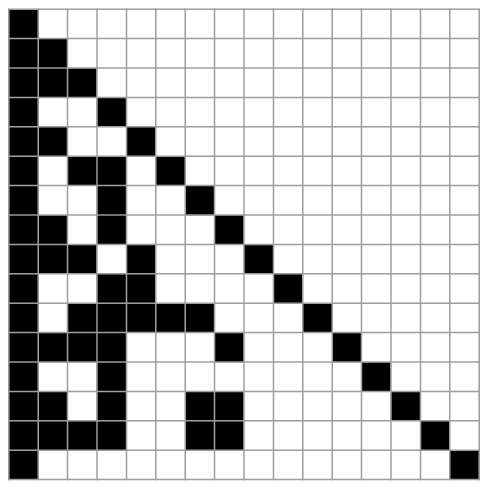}
    \caption{The sets $U_3$ and $U_4$.}
    \label{fig:U3}
\end{figure}

If $u=u_1\cdots u_n$ is a finite word over $\{0,1\}$, we make use of the following convention: $0.u$ has to be understood as the rational number $\sum_{i=1}^n u_i/2^i$.

\begin{remark}\label{rem:pixel} 
Each pair $(u,v)$ of words of length at most $n$ with an odd binomial coefficient gives rise to a square region in $U_n$. More precisely, we have the following situation. 
Let $n\ge 0$ and $u,v\in L_n$ such that $\binom{u}{v}\equiv 1\bmod{2}$. 
We have
$$(0.0^{n-|v|}v,0.0^{n-|u|}u)+Q/2^n\subset U_n$$
as depicted in Figure~\ref{fig:vizu}.
\begin{figure}[h!tbp]
    \centering
\begin{tikzpicture}[line cap=round,line join=round,>=triangle 45,x=1.0cm,y=1.0cm, scale=.6]
\fill[fill=black,fill opacity=0.8] (-1,0) -- (-1,-1) -- (0,-1) -- (0,0) -- cycle;
\draw [->] (-3,5) -- (-3,-4);
\draw [->] (-3,5) -- (6,5);
\draw (-3.94,-3.66) node[anchor=north west] {$u$};
\draw (6,5.5) node[anchor=north west] {$v$};
\draw (-3.49,6) node[anchor=north west] {$0$};
\draw (-3.59,-2.7) node[anchor=north west] {$1$};
\draw (4.7,6) node[anchor=north west] {$1$};
\draw [dash pattern=on 5pt off 5pt] (-3,-3)-- (5,-3);
\draw [dash pattern=on 5pt off 5pt] (5,-3)-- (5,5);
\draw (-1,0)-- (-1,-1);
\draw (-1,-1)-- (0,-1);
\draw (0,-1)-- (0,0);
\draw (0,0)-- (-1,0);
\draw [dash pattern=on 5pt off 5pt] (-3,0)-- (-1,0);
\draw [dash pattern=on 5pt off 5pt] (-1,0)-- (-1,5);
\draw (-1.4,6) node[anchor=north west] {$0.0^{n-|v|}v$};
\draw (-5.7,0.4) node[anchor=north west] {$0.0^{n-|u|}u$};
\end{tikzpicture}
\caption{Visualization of a square region in $U_n$.}
    \label{fig:vizu}
\end{figure}
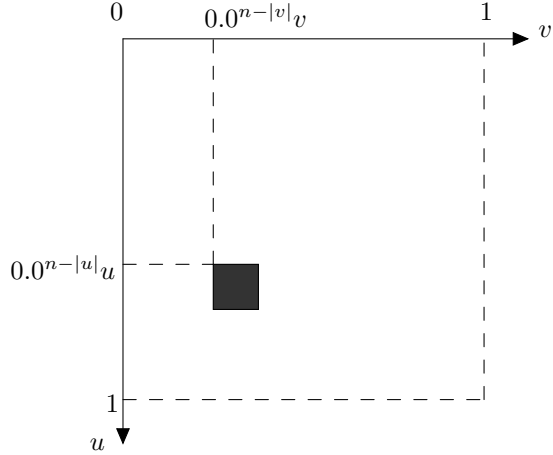
\end{remark}

We will consider the space $(\mathcal{H}(\mathbb{R}^2),d_h)$ of the non-empty compact subsets of $\mathbb{R}^2$ equipped with the Hausdorff metric $d_h$ induced by the Euclidean distance $d$ on $\mathbb{R}^2$. It is well known that $(\mathcal{H}(\mathbb{R}^2),d_h)$ is complete \cite{falconer}. We let $B(x,\epsilon)$ denote the open ball of radius $\epsilon$ centered at $x\in\mathbb{R}^2$ and we denote the $\epsilon$-fattening of a subset $S\subset\mathbb{R}^2$ by 
$$[S]_\epsilon:=\bigcup_{x\in S}B(x,\epsilon).$$

Our aim is to prove that the sequence $(U_n)_{n\ge 0}$ of compact subsets of $[0,1]\times[0,1]$ is converging and we provide an elementary description of the limit set $\mathcal{L}$.

\subsection{The $(\star)$ condition}
Some pairs of words $(u,v)\in L\times L$ have the property that not only $\binom{u}{v}\equiv 1\bmod{2}$ but also $\binom{uw}{vw}\equiv 1\bmod{2}$ for all words $w$. Such a property creates a particular pattern occurring in $U_n$ for all $n\ge |u|$.
\begin{definition}
    Let $(u,v)\in L\times L$. We say that $(u,v)$ satisfies the $(\star)$ condition, if $(u,v)\neq(\varepsilon,\varepsilon)$,
$$\binom{u}{v}\equiv 1\bmod{2},\ \binom{u}{v0}= 0 \text{ and }\binom{u}{v1}=0.$$
In particular, this condition implies that $|v|\le |u|$. 
\end{definition}

Note that if $(u,v)$ satisfies $(\star)$, then we clearly have 
\begin{equation}\label{eq:prolong}
    \binom{u}{vw}=0
\end{equation}
for all non-empty words $w$.

\begin{example}Some pairs $(u,v)$ satisfying $(\star)$:
    $$\begin{array}{r|rrrrr}
u&1&101&1001&1101&1110\\
\hline
v&1& 11&  11& 111&  10\\
\end{array}.$$
\end{example}

\begin{lemma}\label{lem:star-extension}
    If $(u,v)\in L\times L$ satisfies $(\star)$, then both $(u0,v0)$ and $(u1,v1)$ satisfy the $(\star)$ condition.
\end{lemma}

\begin{proof}
    The fact that $\binom{u0}{v0}\equiv 1\bmod{2}$ directly follows from Lemma~\ref{lem:binomial}. If $\binom{u0}{v00}>0$ or $\binom{u0}{v01}>0$ then, $v0$ must appears as a subsequence of $u$ contradicting the assumption.
\end{proof}

\begin{remark}
    Let $(u,v)$ satisfying $(\star)$ such that $|v|\le |u|=\ell$. From Remark~\ref{rem:pixel}, we know that $(0.0^{\ell-|v|}v,0.u)+Q/2^\ell\subset U_\ell$. As a consequence of Lemma~\ref{lem:star-extension}, $(u,v)$, $(u0,v0)$ and $(u1,v1)$ have an odd binomial coefficient and thus correspond to square regions in $U_{\ell+1}$, i.e.,
$$\{(0.0^{\ell+1-|v|}v,0.0u),(0.0^{\ell-|v|}v0,0.u0),(0.0^{\ell-|v|}v1,0.u1)\}+Q/2^{\ell+1}\subset U_{\ell+1}.$$
Iterating this argument yields, for all $n\ge 0$,
$$\bigcup_{|w|=0}^n (0.0^{\ell+n-|w|-|v|}vw,0.0^{n-|w|}uw)+Q/2^{\ell+n}\subset U_{\ell+n}.$$
As an example, consider the pair $(101,11)$ and the corresponding squares in $U_3, U_4$ and $U_5$ as depicted in Figure~\ref{fig:squares}.
\begin{figure}[h!tbp]
    \centering
    \scalebox{.5}{\includegraphics{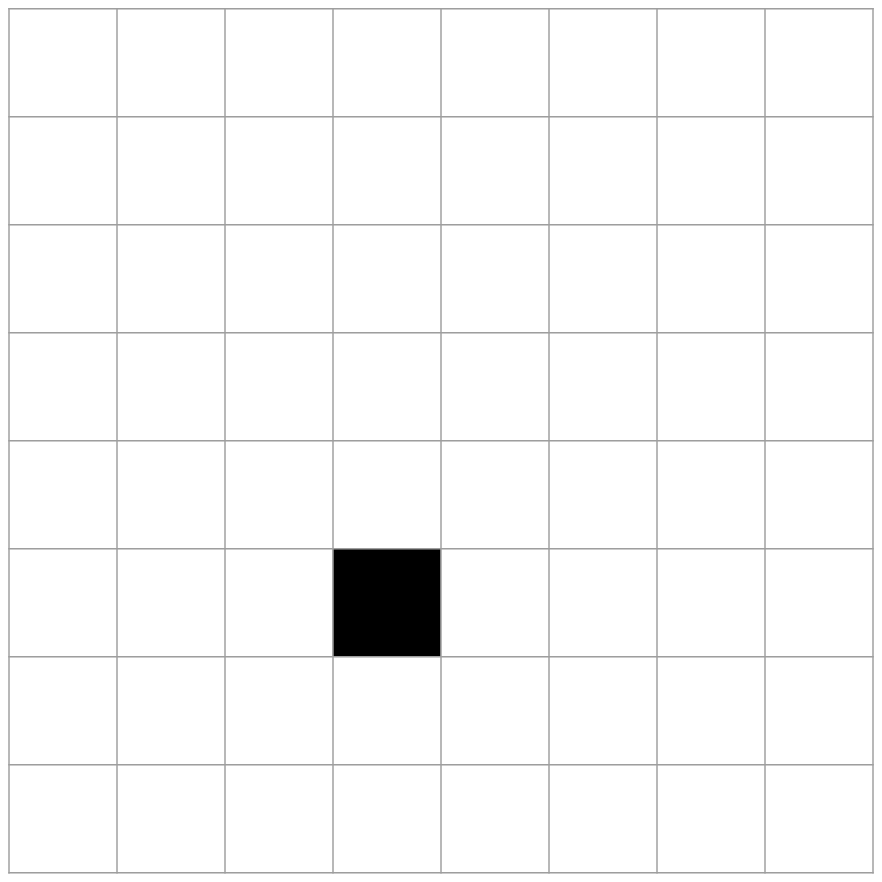}\quad \includegraphics{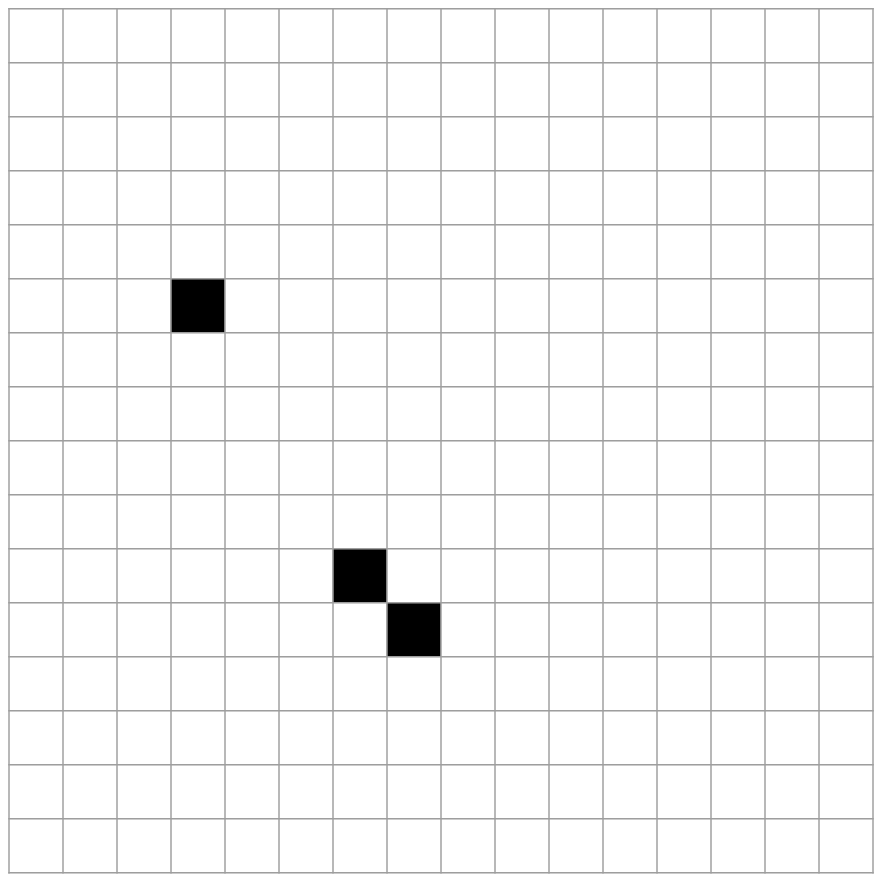}\quad \includegraphics{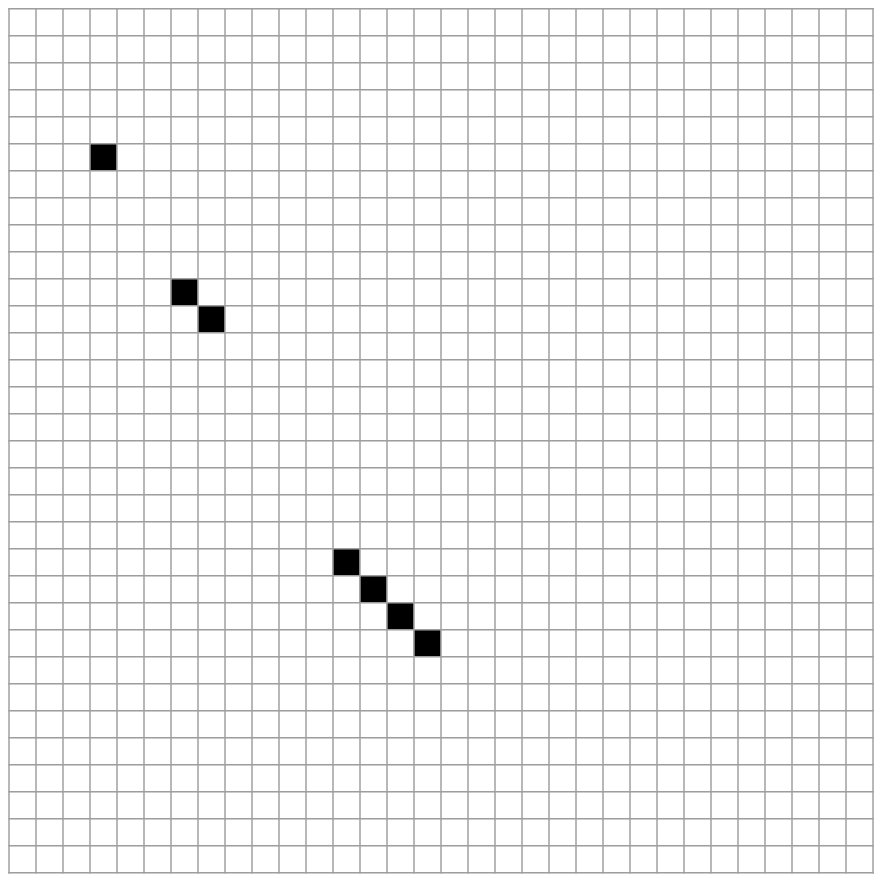}}
    \caption{The pair $(101,11)$ satisfying $(\star)$ in $U_3,U_4,U_5$.}
    \label{fig:squares}
\end{figure}
\end{remark}

\subsection{Definition of an initial set $\mathcal{A}_0$}
Observe that, for $(u,v)$ satisfying $(\star)$, the sequence
$$\left(((0.0^{|u|-|v|}v,0.u)+Q/2^{|u|}) \cap U_n\right)_{n\ge 0}$$
converges to the diagonal of the square $(0.0^{|u|-|v|}v,0.u)+Q/2^{|u|}$. See for instance, in Figure~\ref{fig:squares}, the square region of size $1/8$ and upper-left corner $(3/8,5/8)$: in $U_3$ we have one black square, which is divided into two black squares of size $1/16$ in $U_4$ and then, four black squares of size $1/32$ in $U_5$. This observation leads us to the definition of the set $\mathcal{A}_0$.

\begin{definition}\label{def:seg}
Let $(u,v)$ in $L\times L$ such that $|u|\ge|v|\ge 1$. We define a closed segment $S_{u,v}$ of slope $1$ and length $\sqrt{2}\cdot 2^{-|u|}$ in $[0,1]\times [1/2,1]$. The endpoints of $S_{u,v}$ are given by 
$$A_{u,v}:=(0.0^{|u|-|v|}v,0.u) \text{ and } B_{u,v}:=(0.0^{|u|-|v|}v+2^{-|u|},0.u+2^{-|u|}).$$
 Note that if we allow infinite binary expansions ending with ones, we have
$$B_{u,v}:=(0.0^{|u|-|v|}v111\cdots,0.u111\cdots).$$
Observe that $S_{u,v}$ is included in $[1/2^{|u|-|v|+1},1/2^{|u|-|v|}]\times[1/2,1]$.
\end{definition}

First, let us discuss relative positions of two segments of the form $S_{u,v}$ and in particular, we explain when $S_{s,t}\subset S_{u,v}$.

\begin{lemma}
Let $(u,v)$ and $(s,t)$ in $L\times L$ such that $|u|\ge|v|\ge 1$ and $|s|\ge|t|\ge 1$. The point $A_{s,t}$ belongs to $A_{u,v}+[0,2^{-|u|})\times [0,2^{-|u|})$ if and only if there exist two words $w,z$ of the same length such that $s=uw$ and $t=vz$.
\end{lemma}

\begin{proof}
    It follows from base-$2$ expansions of integers. Observe that, if $|w|\neq |z|$, then $0.0^{|u|-|v|+|w|-|z|}vz$ does not belong to the interval $[0.0^{|u|-|v|}v,0.0^{|u|-|v|}v+2^{-|u|})$.
\end{proof}

\begin{corollary}
    Let $(u,v)$ and $(s,t)$ satisfying $(\star)$. If the point $A_{s,t}$ belongs to $A_{u,v}+[0,2^{-|u|})\times [0,2^{-|u|})$, then $S_{s,t}\subset S_{u,v}$.
\end{corollary}

\begin{proof}
    We know that there exist two words $w,z$ of the same length such that $s=uw$ and $t=vz$. From Lemma~\ref{lem:binomial2}, we get
$$\binom{s}{t}=\binom{uw}{vz}=\sum_{\substack{f,g\in L\\ vz=fg}}\binom{u}{f}\binom{w}{g}.$$
If $|f|>|v|$, then $v$ is a strict prefix of $f$ and thus $\binom{u}{f}=0$ applying \eqref{eq:prolong}. If $|f|<|v|$, then $z$ is a strict suffix of $g$. In particular, in this case $|w|<|g|$ and thus $\binom{w}{g}=0$. We conclude that 
$$\binom{s}{t}=\binom{u}{v}\binom{w}{z}.$$
Since $(s,t)$ satisfies $(\star)$, we have $w=z$ and the conclusion follows.
\end{proof}

\begin{corollary}\label{cor:seg_max}
    Let $(u,v)$ and $(s,t)$ satisfying $(\star)$. Considering the two segments $S_{u,v}$ and $S_{s,t}$, either one is included into the other or, $S_{u,v}\cap S_{s,t}$ is empty or reduced to a common endpoint.
\end{corollary}

\begin{definition}
Let us define the following compact set which is the closure of a countable union of segments
$$\mathcal{A}_0:=\overline{\bigcup_{\substack{(u,v)\\ \text{satisfying} (\star)}} S_{u,v}}.$$
Notice that Definition~\ref{def:seg} implies that $\mathcal{A}_0 \subset [0,1]\times[1/2,1]$.
\end{definition}

\begin{example}
In Figure~\ref{fig:A0}, each pair $(u,v)$ satisfying $(\star)$ with $|u|\le 6$ corresponds to a point. We have represented all the corresponding segments $S_{u,v}$. Corollary~\ref{cor:seg_max} shows that the set of segments has a partial ordering for inclusion, the maximal segments are the diagonals of the dotted squares. For instance, consider the gray square whose upper-left corner corresponds to $(1101,111)$ satisfying $(\star)$. Its diagonal is the segment $S_{1101,111}$ with origin $(7/16,13/16)$ which has length $\sqrt{2}/16$ and contains the four segments $S_{1101w,111w}$ of length $\sqrt{2}/64$ for $|w|=2$.
    \begin{figure}[h!tbp]
        \centering
        \rotatebox{270}{\scalebox{.43}{\includegraphics{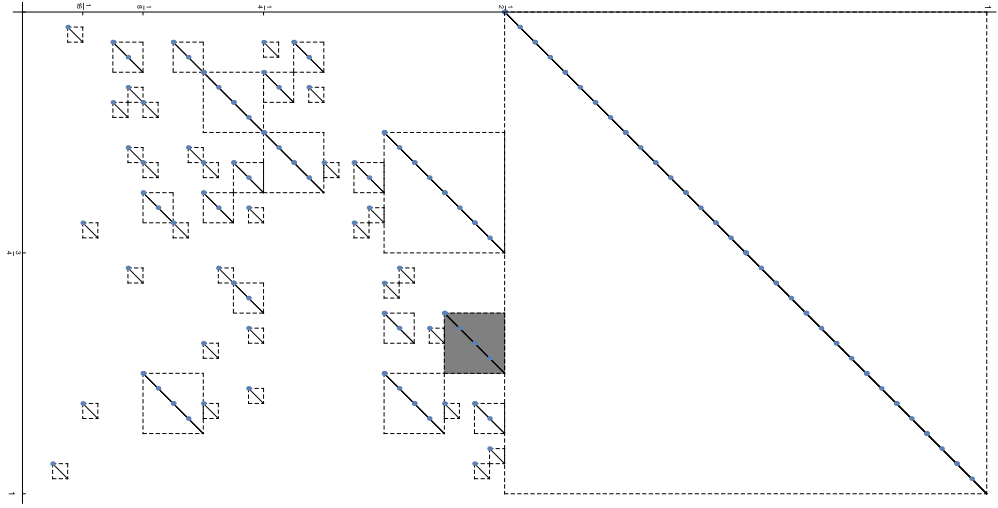}}}
        \caption{An approximation of $\mathcal{A}_0$ computed with words of length $\le 6$.}
        \label{fig:A0}
    \end{figure}
\end{example}

\begin{remark}
In the definition of $\mathcal{A}_0$, we take the closure of a union to ensure the compactness of the set. Here is an example of a limit point that does not belong to the union of segments but to the closure $\mathcal{A}_0$. 
For all $n\ge 0$ and all $r\in \{0,1,\ldots,7\}$, the pair $(10^{8n+4+r}1,10^{8n+r}1)$ satisfies $(\star)$ if and only if $0\le r\le 3$. Indeed, observe that 
$$\binom{10^{8n+4+r}1}{10^{8n+r}1}=\binom{8n+4+r}{8n+r}.$$
Now, we use Theorem~\ref{thm:lucas} and first assume that $0\le r\le 3$. Let $\rep_2(n)=n_k\cdots n_0$ and $\val_2(r_1r_0)=r$ with $r_0,r_1\in\{0,1\}$, then 
$$\binom{8n+4+r}{8n+r}= \binom{n_k}{n_k}\cdots  \binom{n_0}{n_0}\binom{1}{0}\binom{r_1}{r_1}\binom{r_0}{r_0}\equiv 1\bmod{2}.$$
It is easy to see that $(10^{8n+4+r}1,10^{8n+r}1)$ satisfies the other two conditions of $(\star)$. If $7 \ge r\ge 4$, when applying Theorem~\ref{thm:lucas}, the corresponding product contains a factor $\binom{0}{1}$ and the result is thus even. 
Consequently, we have a sequence of segments in $\mathcal{A}_0$ with one endpoint being of the form $(1/32+1/2^m,1/2+1/2^m,)$ with $m\ge 6$  
and the point $(1/32,1/2)$ is an accumulation point of $\mathcal{A}_0$. In Figure~\ref{fig:limit}, we have represented the segments corresponding to $n=0,1$ and $r=0,\ldots,3$.
\begin{figure}[h!t]
    \centering
    \rotatebox{270}{\scalebox{.5}{\includegraphics{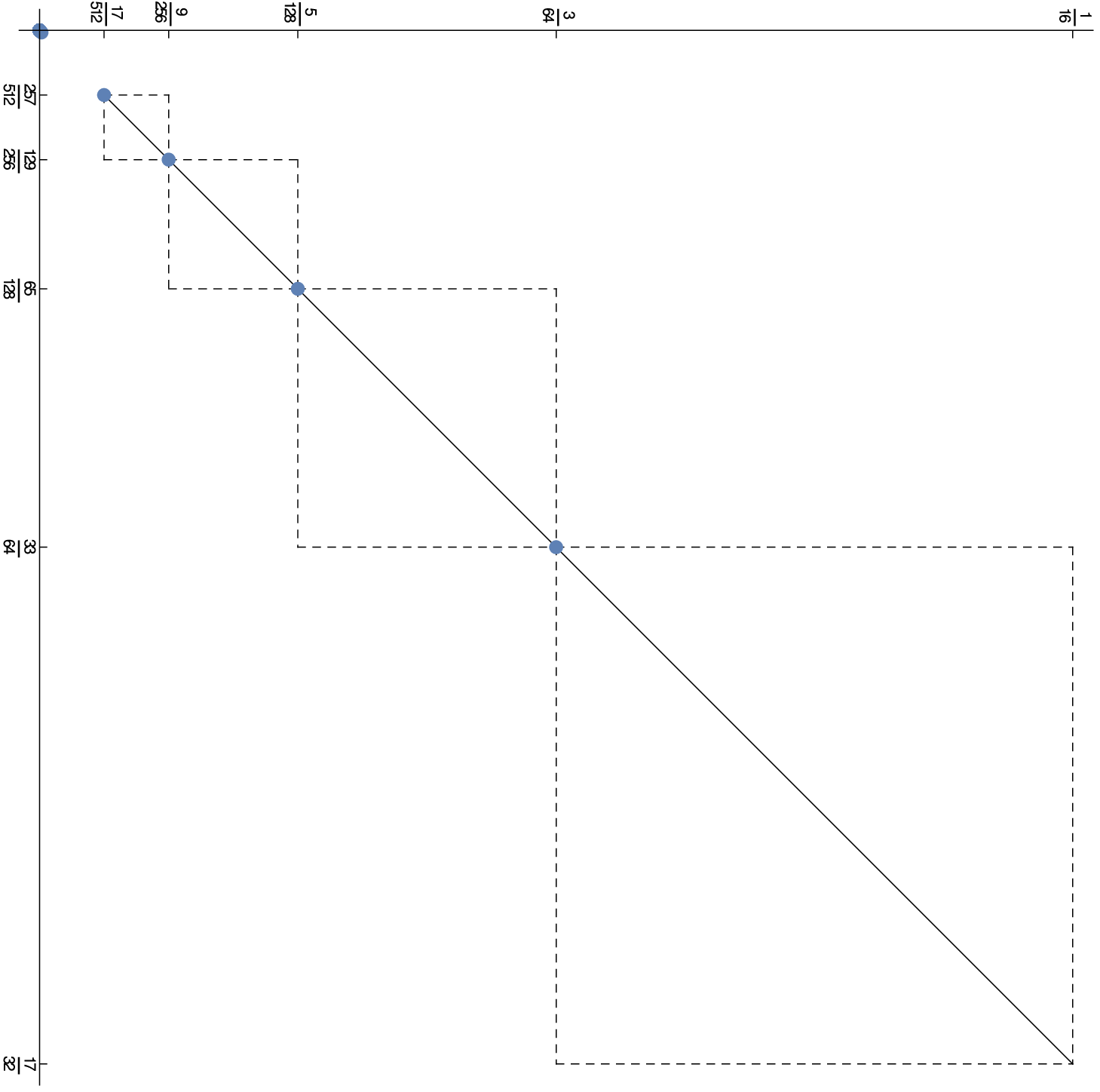}}}
\rotatebox{270}{\scalebox{.5}{\includegraphics{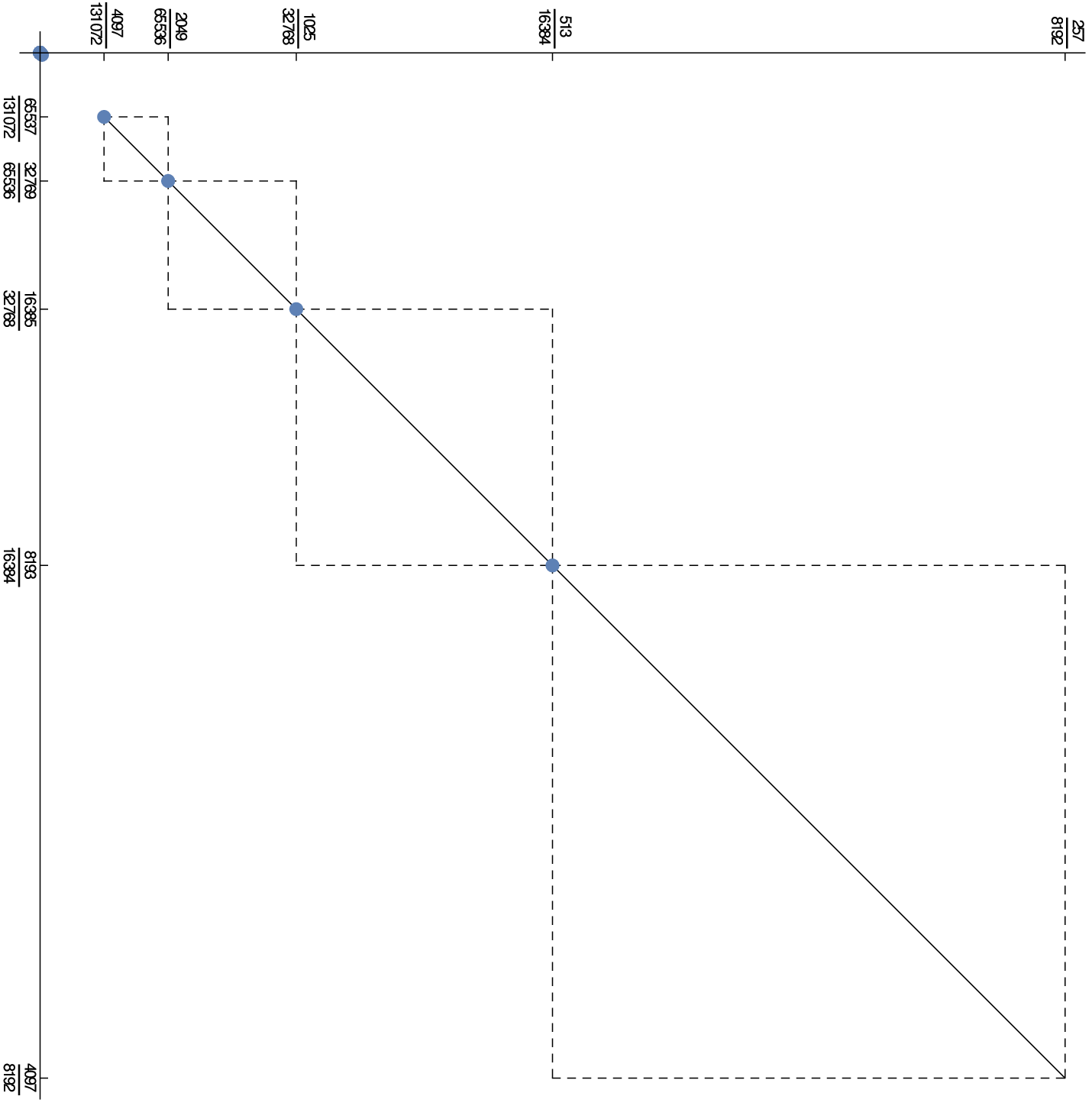}}}
    \caption{A zoom on $\mathcal{A}_0$ in $[17/2^9,1/2^4]\times [27/2^9,17/2^5]$ and in the smaller area $[4097/2^{17},257/2^{13}]\times [65537/2^{17},4097/2^{13}]$.}
    \label{fig:limit}
\end{figure}
\end{remark}

\subsection{Definition of the sequence $(\mathcal{A}_n)_{n\in\mathbb{N}}$}

In the following definition, we introduce another sequence of compact sets obtained by transforming $\mathcal{A}_0$ under iterations of two maps. This new sequence, which is shown to be a Cauchy sequence, will allow us to define properly the limit set $\mathcal{L}$. 

\begin{definition}
We let $c$ denote the homothety of center $(0,0)$ and ratio $1/2$ and consider the map $h:(x,y)\mapsto (x,2y)$. 
Now define a sequence of compact sets 
$$\mathcal{A}_n:=\bigcup_{\substack{0\le i\le n\\ 0\le j\le i}} h^j(c^i(\mathcal{A}_0)).$$
\end{definition}

Let $m,n$ with $m\leq n$. Using Figure~\ref{fig:hc}, observe that
\begin{equation}
    \label{eq:stabilisation}
    \mathcal{A}_m\cap ([1/2^{m+1},1]\times [0,1])=\mathcal{A}_n\cap ([1/2^{m+1},1]\times [0,1]).
\end{equation}
\begin{figure}[h!t]
    \centering
    {\psfrag{A0}{$\mathcal{A}_0$}\psfrag{cA0}{$c(\mathcal{A}_0)$}\psfrag{hcA0}{$h(c(\mathcal{A}_0))$}\psfrag{c2A0}{$c^2(\mathcal{A}_0)$}\psfrag{hc2A0}{$h(c^2(\mathcal{A}_0))$}\psfrag{h2c2A0}{$h^2(c^2(\mathcal{A}_0))$}\psfrag{h}{$h$}\psfrag{c}{$c$}\psfrag{0}{$0$}\psfrag{1}{$1$}
\includegraphics{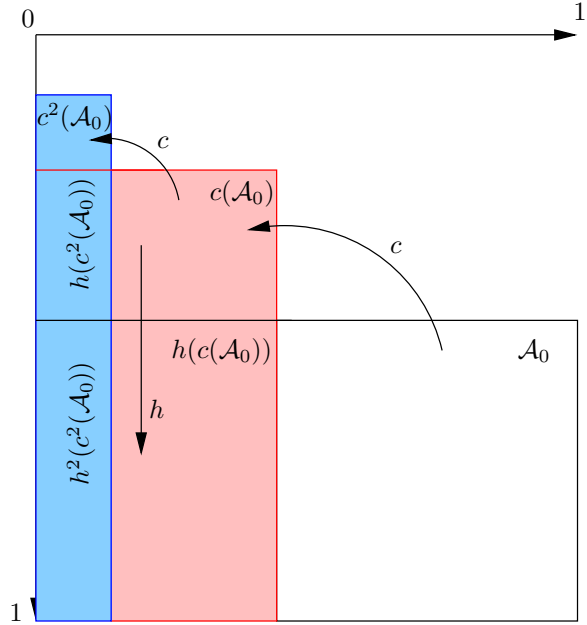}}
    \caption{Two applications of $c$ and $h$ from $\mathcal{A}_0$.}
    \label{fig:hc}
\end{figure}

\begin{example}\label{exa:32}
In Figure~\ref{fig:a3}, we have depicted two original segments in $\mathcal{A}_0$ (in black), then one application of $c$ possibly followed by $h$ (in red), then a second application of $c$ followed by at most $2$ applications of $h$ (in blue). 
    \begin{figure}[h!tbp]
        \centering
        \rotatebox{270}{\scalebox{.4}{\includegraphics{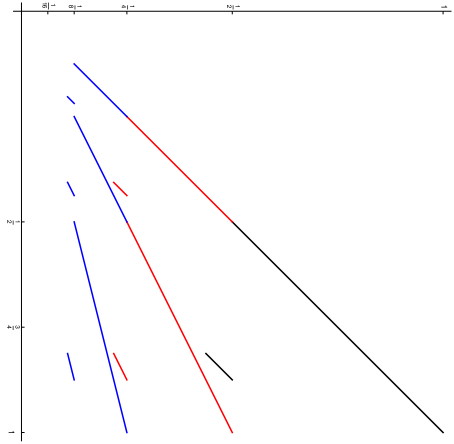}}}
        \caption{A subset of $\mathcal{A}_2$.}
        \label{fig:a3}
    \end{figure}
For instance, the segment $S_{1,1}$ with endpoints $(1/2,1/2)$ and $(1,1)$ belongs to $\mathcal{A}_0$. Thus, the segment $h^p(c^{p+j}(S_{1,1}))$ with endpoints $(1/2^{p+j+1},1/2^{j+1})$ and  $(1/2^{p+j},1/2^{j})$ belongs to $\mathcal{A}_{p+j}$ for all $p\ge 0$ and all $j\ge 0$. Since $\mathcal{A}_p\subset \cdots \subset \mathcal{A}_{2p}$, considering the union of segments of the latter form for $j=0,\ldots,p$, the segment with endpoints $(1/2^{2p+1},1/2^{p+1})$ and $(1/2^{p},1)$ belongs to $\mathcal{A}_{2p}$.
\end{example}

\begin{lemma}\label{lem:cauchy_seq}
    The sequence $(\mathcal{A}_n)_{n\ge 0}$ is a Cauchy sequence.
\end{lemma}

\begin{proof}
Let $\epsilon>0$. Take $N$ such that $1/2^{N/2}<\epsilon$. Let $n>m>N$. From \eqref{eq:stabilisation},  $[\mathcal{A}_m]_\epsilon$ contains $\mathcal{A}_n \cap ([1/2^{m+1},1]\times [0,1])$. Assume first that $m$ is even. 
Since $\mathcal{A}_m$ contains the segment with endpoints $(1/2^{m+1},1/2^{(m/2)+1})$ and $(1/2^{m/2},1)$ (details are given in the previous example), then $[\mathcal{A}_m]_\epsilon$ contains $[0,1/2^{m+1}) \times [0,1]$. If $m$ is odd, $\mathcal{A}_{m-1}$ and also $\mathcal{A}_m$ contain the segment with endpoints $(1/2^{m},1/2^{((m-1)/2)+1})$ and $(1/2^{(m-1)/2},1)$ and the same conclusion follows. Thus $[\mathcal{A}_m]_\epsilon$ contains $\mathcal{A}_n$. Since $\mathcal{A}_m\subset\mathcal{A}_n$, we directly have that $[\mathcal{A}_n]_\epsilon$ contains $\mathcal{A}_m$.
\end{proof}

\begin{definition} 
Since we have a Cauchy sequence in the complete metric space $(\mathcal{H}(\mathbb{R}^2),d_h)$, the limit of $(\mathcal{A}_n)_{n\ge 0}$ is a well defined compact set denoted by $\mathcal{L}$. An approximation of this set is given in Figure~\ref{fig:approx} in the appendix.
\end{definition}

\section{Convergence of $(U_n)_{n\ge 0}$ to $\mathcal{L}$}\label{sec:4}
Now we will show that the sequence $(U_n)_{n\ge 0}$ of compact subsets of $[0,1]\times [0,1]$ converges to $\mathcal{L}$. The first part is to show that, when $\epsilon$ is a positive real number, then $U_n\subset [\mathcal{L}]_\epsilon$ for all sufficiently large $n$. 

\begin{lemma}\label{lem:first_inclusion} Let $\epsilon > 0$. For all $n\in\mathbb{N}$ such that $2^{-n+1}<\epsilon$, we have
    $$U_n\subset [\mathcal{L}]_\epsilon.$$
\end{lemma}

\begin{proof}
    Let $\epsilon>0$. Choose $n$ such that $2^{-n+1}<\epsilon$.

Let $(x,y)\in U_n$. From Remark~\ref{rem:pixel}, there exists $(u,v)\in L\times L$ such that $\binom{u}{v}\equiv 1\bmod{2}$, $|u|\le n$ and $$(x,y)\in (0.0^{n-|v|}v,0.0^{n-|u|}u)+Q/2^n.$$ 

Assume first that $(u,v)$ satisfies $(\star)$. 
The segment $S_{u,v}$ of length $\sqrt{2}\cdot 2^{-|u|}$ having $A_{u,v}=(0.0^{|u|-|v|}v,0.u)$ as endpoint belongs to $\mathcal{A}_0$. Now apply $n-|u|$ times the homothety $c$ to this segment. So the segment $c^{n-|u|}(S_{u,v})$ of length $\sqrt{2}\cdot 2^{-n}$ of endpoint $(0.0^{n-|v|}v,0.0^{n-|u|}u)$ belongs to $\mathcal{A}_{n-|u|}$ and thus to $\mathcal{L}$. Hence $d((x,y),\mathcal{L})<\epsilon$.

Now assume that $(u,v)$ does not satisfy $(\star)$. By assumption, we have an odd number $r$ of occurrences of $v$ in $u$. For each occurrence of $v$ in $u$, we count the total number of zeroes after it. We thus define a sequence of non-negative integer indices 
$$|u| \geq i_1 \geq i_2 \geq \cdots \geq i_r \geq 0$$
corresponding to the number of zeroes following the first, the second, ..., the $r$th occurrence of $v$ in $u$. Now let $k$ be a non-negative integer such that $k > \lceil \log_2 |u|\rceil$.  We get 
$$
\binom{u0^{2^k}1}{v0^{2^k}1} = \sum_{\ell=1}^r \binom{2^k+i_\ell}{2^k}.
$$
Indeed, for each $\ell\in\{1,\ldots,r\}$, consider the $\ell$th occurrence of $v$ in $u$: we have the factorization $u=pw$ where the last letter of $p$ is the last letter of the  $\ell$th occurrence of $v$ and $|w|_0=i_\ell$. With this particular occurrence of $v$, we obtain occurrences of $v0^{2^k}1$ in $u0^{2^k}1$ by choosing $2^k$ zeroes among the $2^k+i_\ell$ zeroes available in $w0^{2^k}1$. Moreover, with the long block of $2^k$ zeroes, it is not possible to have any other occurrence of $v0^{2^k}1$ than those obtained from occurrences of $v$ in $u$.

For each $\ell\in\{1,\ldots,r\}$, we have
$$
\binom{2^k+i_\ell}{2^k} \equiv 1 \bmod{2}
$$
from Theorem~\ref{thm:lucas}. Since $r$ is odd, we get
$$
\binom{u0^{2^k}1}{v0^{2^k}1} \equiv 1 \bmod{2}.
$$
Now, for all $k\in\mathbb{N}$ such that $k > \lceil \log_2 |u|\rceil$, it is easy to check that $(u0^{2^k}1,v0^{2^k}1)$ satisfies $(\star)$. For the sake of clarity, if $k$ is a non-negative integer such that $k > \lceil \log_2 |u|\rceil$, we define
$$ u_k :=  u0^{2^k}1 \quad \text{and} \quad v_k :=  v0^{2^k}1.$$
As in the first part of the proof, the segment $S_{u_k, v_k}$ of length $\sqrt{2}\cdot 2^{-|u|-2^k-1}$ having $A_{u_k,v_k}=(0.0^{|u|-|v|}v0^{2^k}1,0.u0^{2^k}1)$ as endpoint belongs to $\mathcal{A}_0$. Now apply $n-|u|$ times the homothety $c$ to this segment. So the segment $c^{n-|u|}(S_{u_k,v_k})$ of length $\sqrt{2}\cdot 2^{-n-2^k-1}$ of endpoint $(0.0^{n-|v|}v0^{2^k}1,0.0^{n-|u|}u0^{2^k}1)$ belongs to $\mathcal{A}_{n-|u|}$ and thus to $\mathcal{L}$. Hence $d((x,y),\mathcal{L})<\epsilon$.
\end{proof}

Let $\epsilon>0$. It remains to show that, for all sufficiently large $n\in\mathbb{N}$,  $\mathcal{L}\subset [U_n]_\epsilon$. 

\begin{lemma}\label{lem:Un}
Let $\epsilon>0$. For all $(x,y)\in\mathcal{L}$, there exists $N$ such that for all $n\ge N$, $d((x,y),U_n)<\epsilon$.
\end{lemma}

\begin{proof}
    Let $\epsilon>0$ and let $(x,y)\in\mathcal{L}$. Since $(\mathcal{A}_n)_{n\ge 0}$ converges to $\mathcal{L}$, there exists $N_1$ and $(x',y')\in\mathcal{A}_{N_1}$ such that,  
$$d((x,y),(x',y'))<\epsilon/4.$$
By definition of $\mathcal{A}_{N_1}$, there exist $i,j$ such that $0\le j\le i\le N_1$ and $(x_0',y_0')\in\mathcal{A}_0$ such that
$$h^j(c^i((x_0',y_0')))=(x',y').$$
By definition of $\mathcal{A}_0$, there exists a pair $(u,v)\in L\times L$ satisfying $(\star)$ and $(x_0'',y_0'')\in S_{u,v}$ such that
 $$d((x_0',y_0'),(x_0'',y_0''))<\epsilon/4.$$
Notice that, since $j\le i$,  
\begin{eqnarray*}
d((x',y'),h^j(c^i((x_0'',y_0''))))&=&d(h^j(c^i((x_0',y_0'))),h^j(c^i((x_0'',y_0''))))\\
&\le& d((x_0',y_0'),(x_0'',y_0''))<\epsilon/4.
\end{eqnarray*}
Consequently, we get that
$$d((x,y),h^j(c^i( (x_0'',y_0''))))<\epsilon/2.$$

In the second part of the proof, we will show that $d(h^j(c^i((x_0'',y_0''))),U_n)<\epsilon/2$ for all sufficiently large $n$. We will make use of the constants $i,j$ and words $u,v$ given above.

Let $n\ge 0$. Since $(u,v)\in L\times L$ satisfies $(\star)$, the pairs $(uw,vw)$ satisfy $(\star)$ for all words $w$ of length $n,\ldots,n+i$. In particular, if $w$ is a word of length $n$, since $\binom{uw}{vw}\equiv 1\bmod{2}$, applying Lemma~\ref{lem:carres}, at least one of the two binomial coefficients $\binom{uw0}{vw}$, $\binom{uw1}{vw}$ is odd. Iterating this argument $j$ times, we conclude that at least one of the $2^j$ binomial coefficients of the form $\binom{uwz}{vw}$ with $|z|=j$ is odd. Otherwise stated, at least one of the square regions
$$(0.0^{i+|u|-|v|}vw,0.0^{i-j}uwz)+Q/2^{n+i+|u|}, \text{ with }|z|=j,$$
is a subset of $U_{n+i+|u|}$. We observe that each of these square regions is intersected by $h^j(c^i(S_{u,v}))$. Indeed, the latter segment has $(0.0^{i+|u|-|v|}v,0.0^{i-j}u)$ and $(0.0^{i+|u|-|v|}v111\cdots,0.0^{i-j}u111\cdots)$ as endpoints and slope $2^j$. This can be visualized in Figure~\ref{fig:segments} where each rectangular gray region contains at least one square region from $U_{n+i+|u|}$.
\begin{figure}[h!t]
    \centering
    \scalebox{.75}{\psfrag{u0n}{$u0^n$}\psfrag{v0n}{$v0^n$}\psfrag{u0n1}{$u0^{n+1}$}\psfrag{v0n1}{$v0^{n+1}$}\psfrag{u0n2}{$u0^{n+2}$}\psfrag{v0n2}{$v0^{n+2}$}\psfrag{u1n2}{$u1^{n+2}$}\psfrag{v1n2}{$v1^{n+2}$}
\psfrag{u1n1}{$u1^{n+1}$}\psfrag{v1n1}{$v1^{n+1}$}
\psfrag{u1n}{$u1^{n}$}\psfrag{v1n}{$v1^{n}$}
\psfrag{Suv}{$S_{u,v}$}\psfrag{0}{$0$}\psfrag{1}{$1$}\psfrag{2}{$2^{-(n+i+|u|)}$}
\psfrag{cSuv}{$c(S_{u,v})$}\psfrag{hcSuv}{$h(c(S_{u,v}))$}\psfrag{c2Suv}{$c^2(S_{u,v})$}\psfrag{hc2Suv}{$h(c^2(S_{u,v}))$}\psfrag{h2c2Suv}{$h^2(c^2(S_{u,v}))$}
\includegraphics{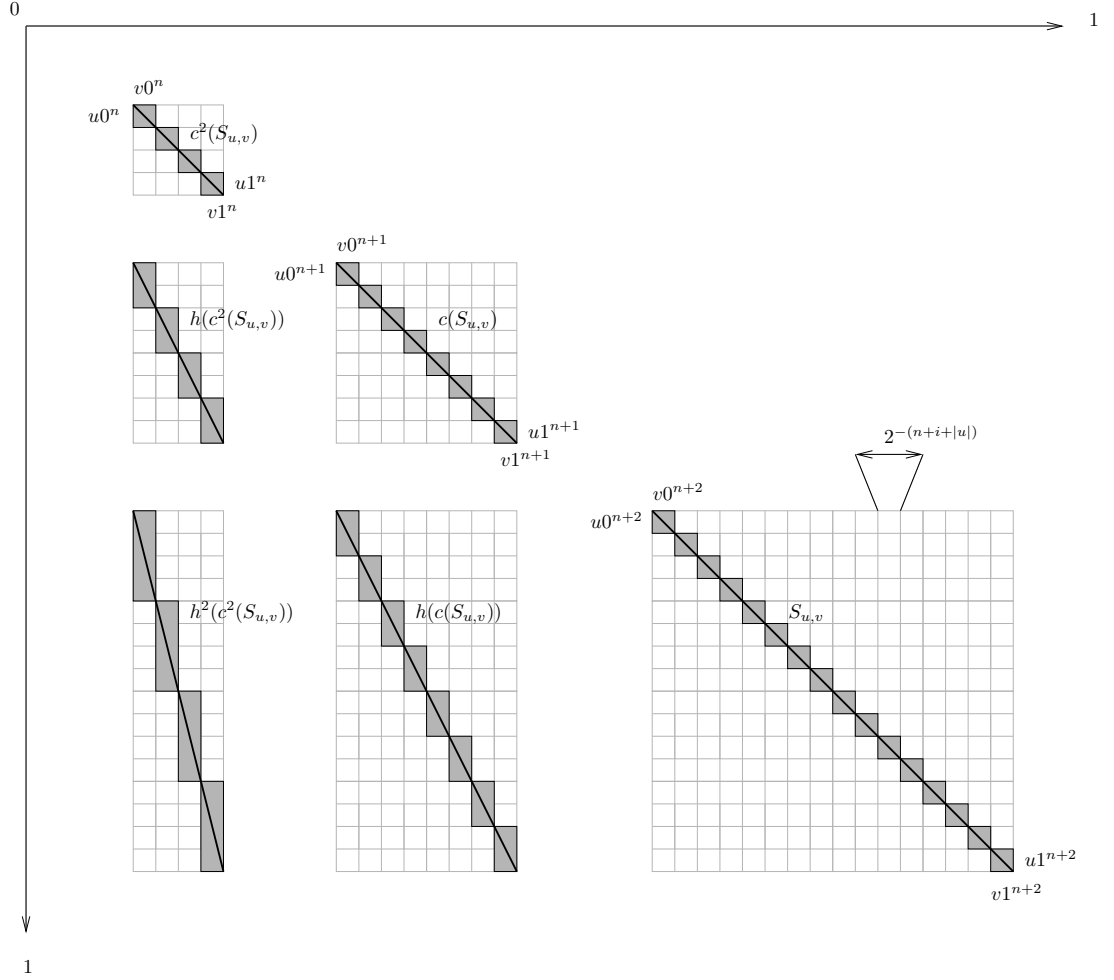}}
    \caption{Situation occurring in the proof of Lemma~\ref{lem:Un}.}
    \label{fig:segments}
\end{figure}
Consequently, every point\footnote{For every point of $h^j(c^i(S_{u,v}))$, this observation will permit us to build a sequence $((f_n,g_n))_{n\ge 0}$ converging to it and such that $(f_n,g_n)\in U_n$, for all $n$, and the distance between two consecutive elements of the sequence tends to zero when $n$ tends to infinity.} of $h^j(c^i(S_{u,v}))$ is at distance at most $2^j/2^{n+i+|u|}$ from a point in $U_{n+i+|u|}$. In particular, this holds for $h^j(c^i((x_0'',y_0'')))$. We choose $N_2$ such that $2^j/2^{N_2+i+|u|}<\epsilon/2$. Hence, for all $n\ge N_2+i+|u|$, 
$$d(h^j(c^i((x_0'',y_0''))),U_{n})<\epsilon/2.$$
To conclude the proof, for all $n\ge N_2+i+|u|$, we have
$$d((x,y),U_{n})<\epsilon.$$
\end{proof}

\begin{theorem}
    The sequence $(U_n)_{n\ge 0}$ converges to $\mathcal{L}$.
\end{theorem}

\begin{proof}
Let $\epsilon>0$. From Lemma~\ref{lem:first_inclusion}, it suffices to show that $\mathcal{L}\subset [U_n]_\epsilon$ for all sufficiently large $n\in\mathbb{N}$. For all $(x,y)\in\mathcal{L}$, from the proof of Lemma~\ref{lem:Un}, there exists a sequence $((f_i(x,y),g_i(x,y))_{i\ge 0}$ such that $(f_i(x,y),g_i(x,y))\in U_i$, for all $i$, and there exists $N_{(x,y)}$ such that, for all $i,j\ge N_{(x,y)}$,
\begin{equation}
    \label{eq:1ein}
d((f_i(x,y),g_i(x,y)),(f_j(x,y),g_j(x,y)))<\epsilon/2    
\end{equation}
and $$d((f_i(x,y),g_i(x,y)),(x,y))<\epsilon/2.$$
We trivially have
$$\mathcal{L}\subset\bigcup_{(x,y)\in\mathcal{L}}B((f_{N_{(x,y)}}(x,y),g_{N_{(x,y)}}(x,y)),\epsilon/2).$$
Since $\mathcal{L}$ is compact, we can extract a finite covering: there exist $(x_1,y_1),\ldots,(x_k,y_k)$ in $\mathcal{L}$ such that
$$\mathcal{L}\subset\bigcup_{j=1}^kB((f_{N_{(x_j,y_j)}}(x_j,y_j),g_{N_{(x_j,y_j)}}(x_j,y_j)),\epsilon/2).$$
Let $N=\max_{j=1,\ldots,k} N_{(x_j,y_j)}$. From \eqref{eq:1ein}, we deduce that, for all $j\in\{1,\ldots,k\}$ and all $n\ge N$, 
$$B((f_{N_{(x_j,y_j)}}(x_j,y_j),g_{N_{(x_j,y_j)}}(x_j,y_j)),\epsilon/2) \subset B((f_{n}(x_j,y_j),g_{n}(x_j,y_j),\epsilon)$$
and therefore
$$\mathcal{L}\subset\bigcup_{j=1}^kB((f_{n}(x_j,y_j),g_{n}(x_j,y_j)),\epsilon)\subset [U_n]_\epsilon.$$
\end{proof}

\section{Extension modulo $p$}\label{sec:5}

In this paper, for the sake of simplicity, we have only considered odd binomial coefficients. It is straightforward to adapt our reasonings, constructions and results to a more general setting. Let $p$ be a fixed prime and $r\in\{1,\ldots,p-1\}$. We can extend Definition~\ref{def:Tn} to 
$$
T_{n,r}:= \bigcup \left\{(\val_2(v),\val_2(u))+Q\mid u,v\in L_n, \binom{u}{v}\equiv r\bmod{p}\right\}\subset [0,2^n]\times [0,2^n]
$$
and introduce corresponding sets $U_{n,r}$ as in Definition~\ref{def:Un}. Since we make use of Lucas' theorem, we limit ourselves to congruences modulo a prime. We just sketch the main differences with the case $p=2$. See, for instance, Figure~\ref{fig:U72} for the case $p=3$ and $r=2$.

In this setting, analogously to Lemma~\ref{lem:carres}, one can observe that if $\binom{u}{vb}\equiv r\bmod{p}$ for some $b\in\{0,1\}$ then there exists $a\in\{0,1\}$ such that $\binom{ua}{vb}\equiv r\bmod{p}$. This observation is useful to adapt the proof of Lemma~\ref{lem:Un}.

The $(\star)$ condition becomes $(\star)_r$
$$\binom{u}{v}\equiv r\bmod{p},\ \binom{u}{v0}= 0 \text{ and }\binom{u}{v1}=0$$
and Lemma~\ref{lem:star-extension} still holds. Note that the pairs $(u,v)$ satisfying this condition depend on the choice of $p$ and $r$. The sets $\mathcal{A}_n$ are defined as before.

\begin{remark}
    The pair $(1,1)$ satisfies $(\star)_r$ if and only if $r=1$. Thus, the segment with endpoints $(1/2^{2m+1},1/2^{m+1})$ and $(1/2^{m},1)$ belongs to $\mathcal{A}_{2m}$ only if $r=1$. This observation made in Example~\ref{exa:32} was used in the proof of Lemma~\ref{lem:cauchy_seq}.
\end{remark}
We present an alternative proof of Lemma~\ref{lem:cauchy_seq} in this general setting.
\begin{proof}[Proof of Lemma~\ref{lem:cauchy_seq} (generalization)]
Let $\epsilon >0$. 
Choose $N$ such that $2^{-N} < \epsilon$. 
Let $m,n$ such that $n>m>N$. 
From \eqref{eq:stabilisation},  $[\mathcal{A}_m]_\epsilon$ contains $\mathcal{A}_n \cap ([1/2^{m+1},1]\times [0,1])$. 
It remains to show that $[\mathcal{A}_m]_\epsilon$ contains $\mathcal{A}_n\cap [0,1/2^{m+1}]\times [0,1]$. 
The pairs $(u0^{|u|}1^r0, u0^{|u|}10)$ satisfy $(\star)_r$ for all words $u\in 1\{0,1\}^*$. 
In this case, the segment of endpoints $(0.0^{r-1}u0^{|u|}10,0.u0^{|u|}1^r0)$ and $(0.0^{r-1}u0^{|u|}10 + 2^{2|u|+r+1},0.u0^{|u|}1^r0 + 2^{2|u|+r+1})$ has length $\sqrt{2} \; 2^{-(2|u|+r+1)}$ and belongs to $\mathcal{A}_0$. 
If we apply $m$ times $c$ and $j\in\{0,\ldots,m\}$ times $h$ to the latter segment, we obtain a segment $S(u,j)$ with endpoint 
$$
	A(u,j):=(0.0^{m+r-1}u0^{|u|}10,0.0^{m-j}u0^{|u|}1^r0)
$$ 
and of length less than $\sqrt{2} \; 2^{-(2|u|+r+1)}$ that belongs to $\mathcal{A}_m$. 
Note that this segment $S(u,j)$ lies into $R(j):=[1/2^{m+r}, 1/2^{m+r-1}]\times [1/2^{m-j+1},1/2^{m-j}]$. 
Observe that
$$
	\bigcup_{u\in 1\{0,1\}^*} A(u,j)
$$
is the diagonal of the region $R(j)$. 
Consequently, for all $j\le m$, we have
$$
	\bigcup_{u\in 1\{0,1\}^*} [S(u,j)]_\epsilon\supset [0,1/2^{m+1}]\times [1/2^{m-j+1},1/2^{m-j}]
$$
since $2^{-(m+r-1)} \leq 2^{-N} < \epsilon$.
Letting $j$ range from $0$ to $m$, we deduce that $[\mathcal{A}_m]_\epsilon$ contains $[0,1/2^{m+1}]\times [0,1]$ and therefore $\mathcal{A}_n \cap ([0,1/2^{m+1}]\times [0,1])$. 
Since $\mathcal{A}_m\subset\mathcal{A}_n$, we directly have that $[\mathcal{A}_n]_\epsilon$ contains $\mathcal{A}_m$.

%
%
\end{proof}

The proof of Lemma~\ref{lem:first_inclusion} follows the same lines. Simply replace the word $u0^{2^k}1$ (resp. $v0^{2^k}1$) with $u0^{p^k}1$ (resp. $v0^{p^k}1$). We may apply Lucas' theorem with base-$p$ expansions.

\section{Appendix}\label{sec:6}

\begin{example}\label{exa:approx}
We have represented in Figure~\ref{fig:approx} the $1370$ segments of $\mathcal{A}_0$ for words of length at most $8$ (we are missing segments of length $\le\sqrt{2}/2^{9}$) and we have applied the maps $h^j(c^i(\cdot))$ to this set of segments for $0 \le j\le i\le 4$. Thus we have an approximation of $\mathcal{A}_4$. Excepting the segments of length $\le\sqrt{2}/2^{9}$ and their images, we have an exact image of $\mathcal{L}$ inside $[1/32,1]\times [0,1]$.
\begin{figure}[h!tb]
   \centering
   \rotatebox{270}{\scalebox{.5}{\includegraphics{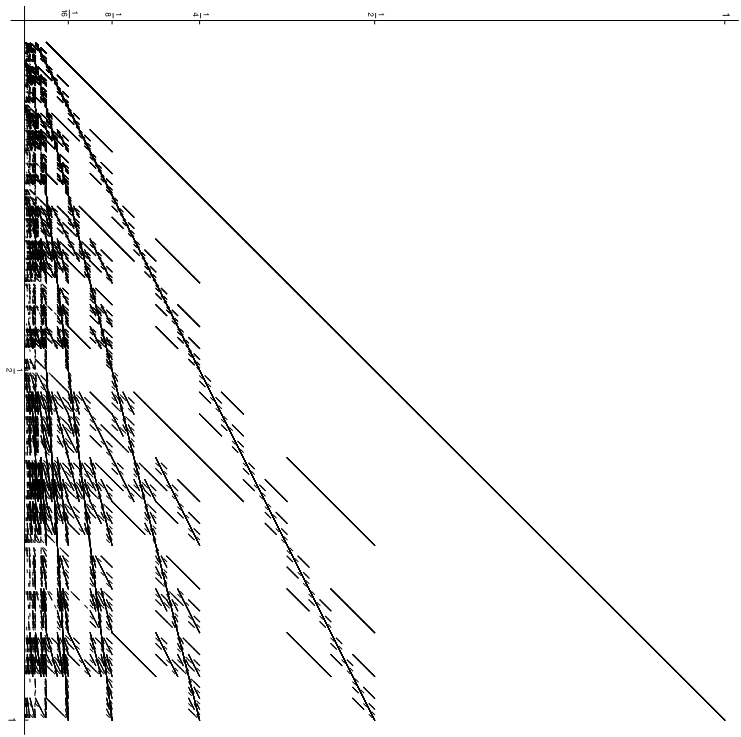}}}
   \caption{An approximation of the limit set $\mathcal{L}$.}
   \label{fig:approx}
\end{figure}
\end{example}

\begin{example}
We have represented the set $U_9$ in Figure~\ref{fig:U9}.
\begin{figure}[h!tb]
   \centering
   {\scalebox{.4}{\includegraphics{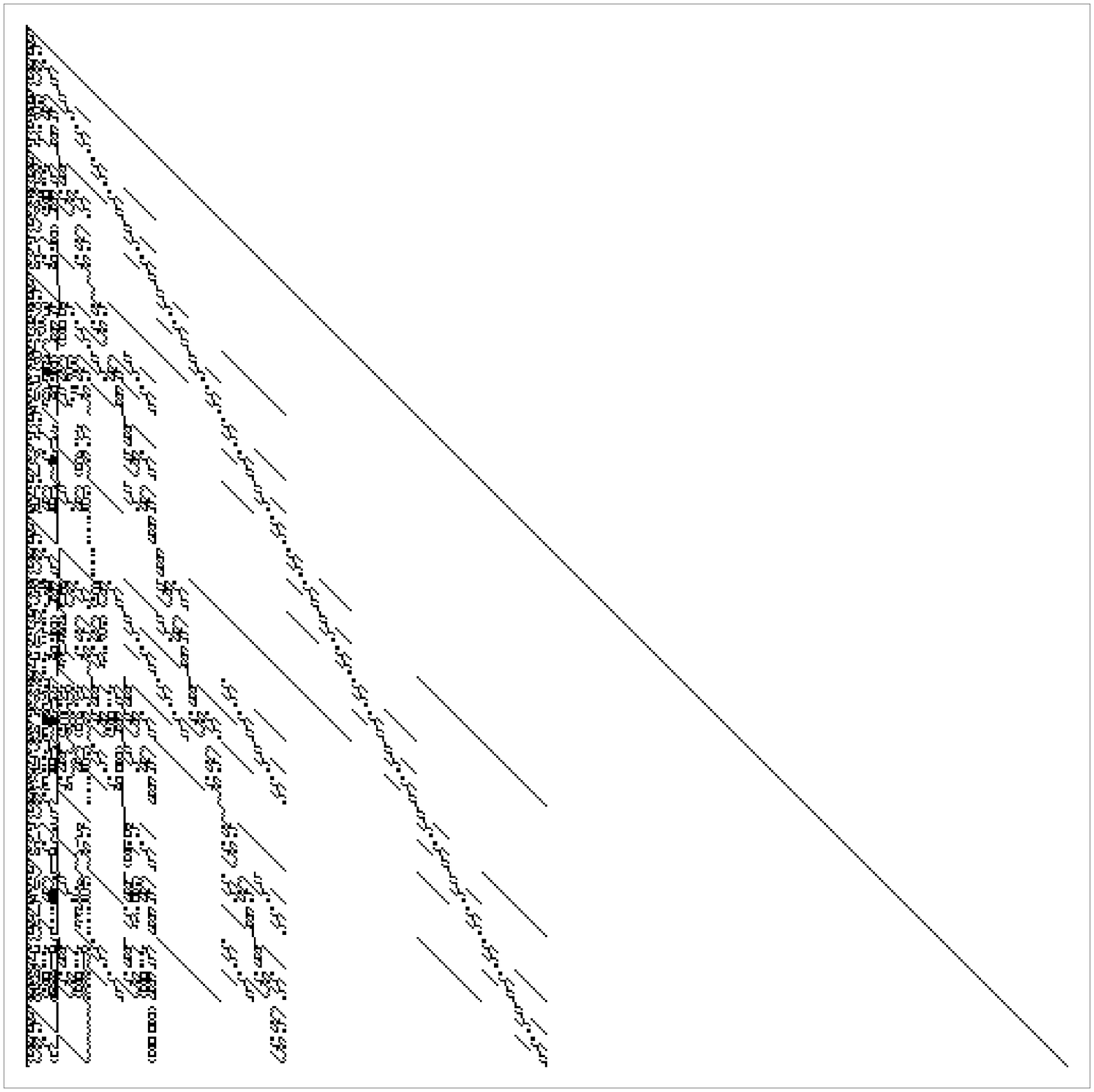}}}
   \caption{The set $U_9$.}
   \label{fig:U9}
\end{figure}
\end{example}

\begin{example}
We have represented in Figure~\ref{fig:U72} the set $U_{7,2}$ when considering binomial coefficients congruent to $2$ modulo $3$ and an approximation of the limit set $\mathcal{L}$ proceeding as in Example~\ref{exa:approx}.
\begin{figure}[h!tb]
   \centering
   \begin{minipage}{.45\linewidth}
       \scalebox{.2}{\includegraphics{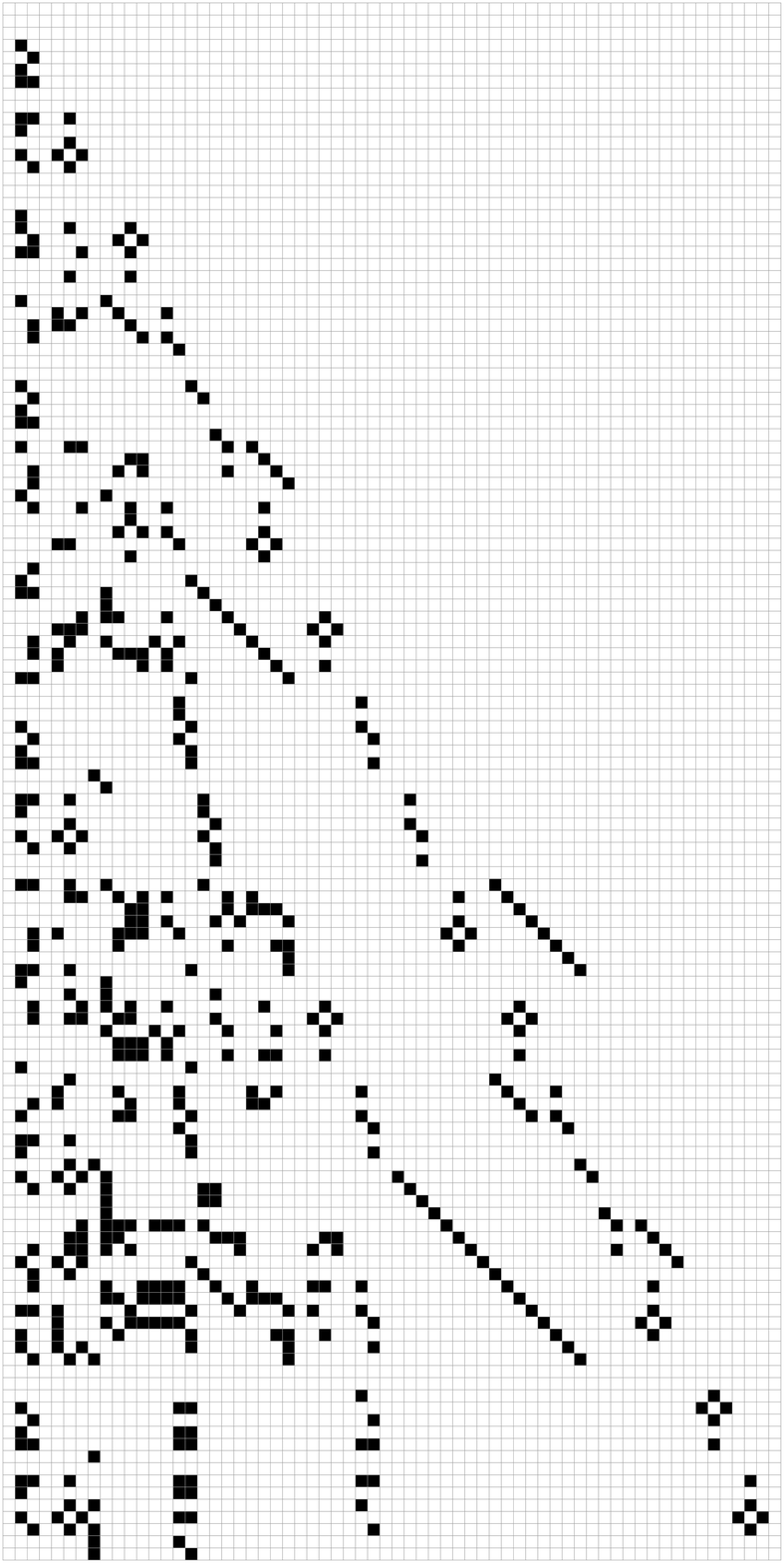}}
   \end{minipage}
    \quad \begin{minipage}{.5\linewidth}
       \rotatebox{270}{\scalebox{.55}{\includegraphics{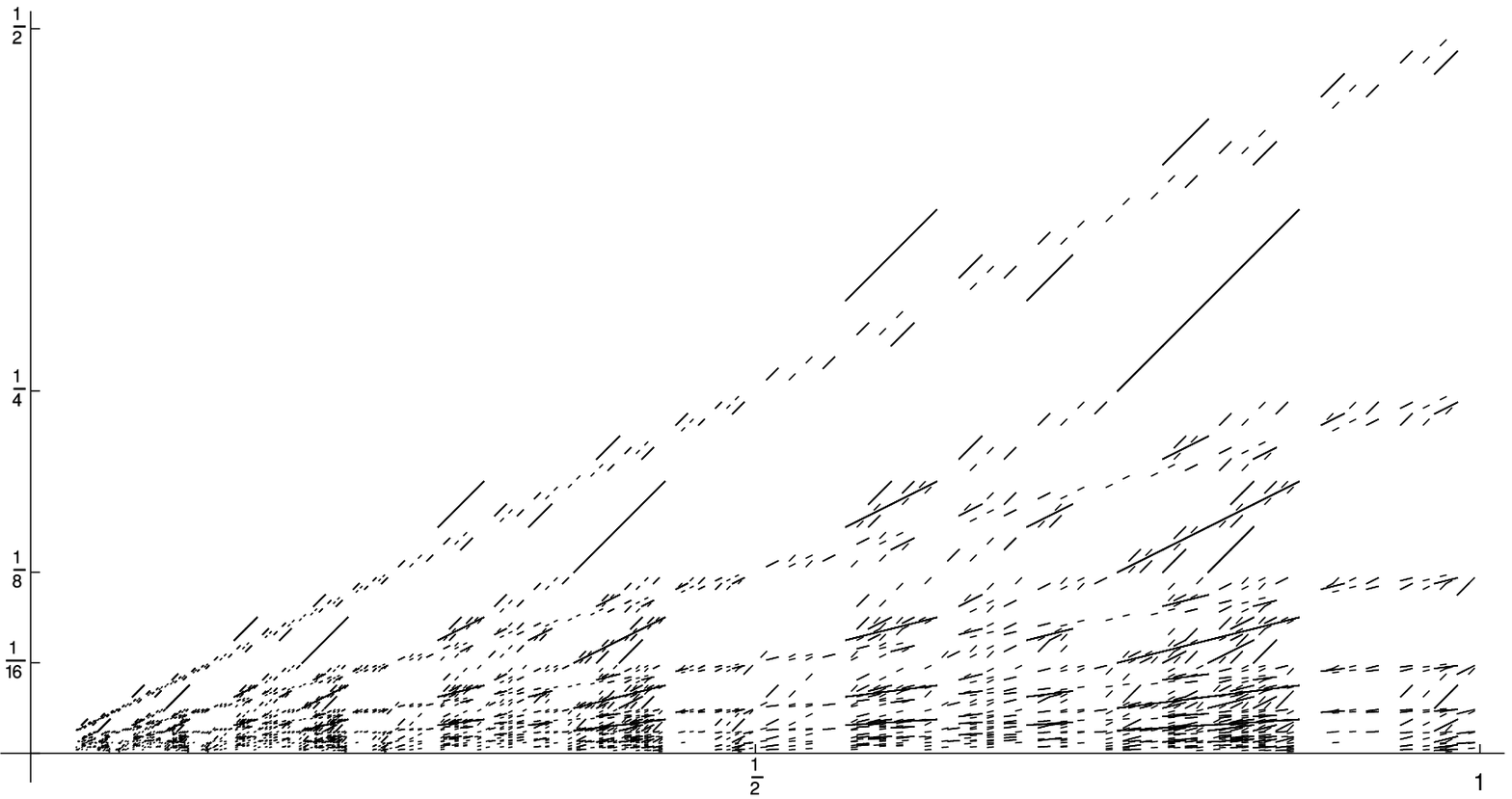}}}
   \end{minipage}
   \caption{The set $U_{7,2}$ and an approximation of the corresponding set $\mathcal{L}$.}
   \label{fig:U72}
\end{figure}
\end{example}

\end{document}